\documentclass[12pt,reqno]{amsart}
\usepackage[T1]{fontenc}
\usepackage[latin1]{inputenc}
\usepackage{latexsym}
\usepackage[arrow, matrix, curve]{xy}
\usepackage[dvips]{graphicx}
\usepackage{epsfig}
\usepackage{bm}
\usepackage{amssymb}
\usepackage{amsmath}
\usepackage{verbatim}
\usepackage{amsthm}
\usepackage{enumerate}
\usepackage{mathrsfs}
\usepackage{hyperref}
\usepackage[hmargin=3cm,vmargin=3cm]{geometry}
\newtheorem{theorem}{Theorem}[section]
\newtheorem{proposition}[theorem]{Proposition}
\newtheorem{lemma}[theorem]{Lemma}
\newtheorem{corollary}[theorem]{Corollary}

\newtheorem*{hypothesis}{Central Hypothesis}

\theoremstyle{definition}

\theoremstyle{remark}

\numberwithin{equation}{section}

\def\ba{{\mathbf a}} 
\def\bb{{\mathbf b}}

 \def\bF{{\mathbf F}}
\def\bg{{\mathbf g}} \def\bG{{\mathbf G}}
\def\bh{{\mathbf h}} 
 
\def\bj{{\mathbf j}}

\def\bx{{\mathbf x}} 
\def\by{{\mathbf y}}

\def\calA{{\mathcal A}} 
\def\calB{{\mathcal B}} 
\def\calC{{\mathcal C}}

\def\calR{{\mathcal R}}

\def\calV{{\mathcal V}}

\def\calZ{{\mathcal Z}}

\def\C{\mathbb C}

\def\N{\mathbb N}

\def\Q{\mathbb Q}
\def\R{\mathbb R}

\def\Z{\mathbb Z}

\def\frJ{{\mathfrak J}}

\def\frm{{\mathfrak m}}\def\frM{{\mathfrak M}}
\def\frn{{\mathfrak n}}\def\frN{{\mathfrak N}}
\def\frP{{\mathfrak P}}
\def\frs{{\mathfrak s}}\def\frS{{\mathfrak S}}

\def\alp{{\alpha}}  
\def\bet{{\beta}}
\def\gam{{\gamma}} \def\Gam{{\Gamma}} 
\def\del{{\delta}} \def\Del{{\Delta}}
\def\eps{\varepsilon}

\def\tet{{\vartheta}}  

\def\kap{{\kappa}}
\def\lam{{\lambda}}

\def\sig{{\sigma}} 
 
\def\vphi{{\varphi}}
\def\ome{{\omega}}

\def\balp{{\boldsymbol \alpha}}
\def\bbet{{\boldsymbol \beta}}

\def\bdel{\boldsymbol \delta}

\def\bxi{{\boldsymbol \xi}}

\def\bpsi{{\boldsymbol \psi}}\def\bPsi{{\boldsymbol \Psi}}

\def\d{{\partial}}

\def\le{\leqslant} \def\ge{\geqslant}

\def\d{{\,{\rm d}}}

\DeclareMathOperator{\rk}{rank}

\DeclareMathOperator{\vol}{vol}

\title[Diagonal forms on hypersurfaces]{The Hasse principle for diagonal forms restricted to lower-degree hypersurfaces}

\author[J. Brandes]{Julia Brandes}
\address{JB: Mathematical Sciences, University of Gothenburg and Chalmers University of Technology, 412 96 G{\"o}teborg, Sweden}
\email{brjulia@chalmers.se}

\author[S. T. Parsell]{Scott T. Parsell}
\address{STP: Department of Mathematics, West Chester University, 25 University Ave., West Chester, PA 19383, U.S.A.}
\email{sparsell@wcupa.edu}


\begin{document}
	
	\begin{abstract}
		We establish the analytic Hasse principle for Diophantine systems consisting of one diagonal form of degree $k$ and one general form of degree $d$, where $d$ is smaller than $k$. By employing a hybrid method that combines ideas from the study of general forms with techniques adapted to the diagonal case, we are able to obtain bounds that grow exponentially in $d$ but only quadratically in $k$, reflecting the growth rates typically obtained for both problems separately. We also discuss some of the most interesting generalisations of our approach.
	\end{abstract}
	\maketitle

\section{Introduction}

The study of systems of Diophantine equations by the circle method has a long and deep history. In particular, thanks to the sustained efforts that have been directed at some of the most relevant special cases of such systems, Waring's problem and Vinogradov's mean value theorem, we now have a fairly good understanding of the conditions under which a given system of diagonal equations has non-trivial solutions in the integers.

In the work at hand, we are especially interested in systems consisting of two Diophantine equations with different degrees. The study of such systems has been initiated by Wooley in his series of papers \cite{W:91sae1,W:91sae2,W:90} on systems consisting of one cubic and one quadratic diagonal equation. Before describing his and subsequent work in greater detail, it is convenient to introduce some notation. For any system of forms $\bF = (F_1, \dots, F_r)$ with $F_1, \dots, F_r \in \Z[x_1, \dots, x_s]$ we write $N_{\bF}(X)$ for the number of integer points $\bx \in [-X,X]^s$ satisfying $F_i(\bx)=0$ simultaneously for $1 \le i \le r$. Heuristic arguments indicate that when $s$ is large enough and the forms satisfy a suitable non-singularity condition, one can expect $N_{\bF}(X)$ to satisfy an asymptotic formula of the shape
\begin{align}\label{1.1}
	N_{\bF}(X) = X^{s-K} (\calC + O(X^{-\nu})),
\end{align}
where $K$ is the total degree of $\bF$. Here, $\calC$ is a non-negative constant and $\nu$ is a small positive real number, both depending at most on the forms $\bF$.

In the case when $F, G \in \Z[x_1, \dots, x_s]$ are non-singular homogeneous polynomials with respective degrees $k$ and $d$, we write $\frs_0(d;k)$ for the least number $\frs_0$ with the property that \eqref{1.1} is satisfied whenever $s>\frs_0$. Similarly, we set $\frs_1(d;k)$ for the least number having the property that $N_{F,G}(X) > 0$ whenever $s>\frs_1$ and the system
\begin{align}\label{1.2}
	F(\bx) = G(\bx)=0
\end{align}
has non-singular solutions in all completions of $\Q$.

Focussing on the case when both $F$ and $G$ are diagonal, Wooley in his original work \cite{W:91sae2} showed that $\frs_1(2;3) \le 13$, a bound that he subsequently \cite{W:98sae4} improved to $\frs_1(2;3) \le 12$. Moreover, his work \cite{W:91sae1} establishes $p$-adic solubility whenever $s \geqslant 11$. Extending Wooley's ideas to higher degrees, the second author \cite{P:02pae} provided bounds for $\frs_1(d;k)$ for a number of small degrees $d$ and $k$. In the special case when $d=1$, Br\"udern and Robert \cite{BR:15} showed that $\frs_0(1;k) \le 2^k+1$, and thus in particular $\frs_0(1;3)\le 9$. All but the last one of these bounds, meanwhile, have been rendered obsolete with the arrival of Wooley's efficient congruencing method \cite{W:12ec, W:16vmvt3, W:17mec2, W:19nec} and the decoupling estimates of Bourgain, Demeter and Guth \cite{BDG:16}, which gave rise to the stronger bound $\frs_0(d;k) \le k(k+1)$ for any pair of distinct degrees $(d,k)$ with $k>d$, and in particular $\frs_0(2;3) \le 32/3$ (see \cite{W:15sae5}). The present authors \cite{B:17comp}, \cite{BP:17} have undertaken some work to extend these estimates to systems consisting of several quadratic and cubic forms.
 
The picture is far less satisfactory when one of the forms $F$ and $G$ fails to be diagonal. In this situation, one has to make recourse to the much more general work of Browning and Heath-Brown \cite{BHB:17} on forms of differing degrees, in which they investigate systems of Diophantine equations in full generality. Specialised to the case of two forms with degrees $k$ and $d$ where $k>d$, their methods produce the bound
\begin{align}\label{1.3}
	\frs_0(d;k) \le (2+d) (k-1) 2^{k-1} + d 2^{d-1}
\end{align}
(see \cite[Corollary~1.5]{BHB:17}). In particular, we have exponential growth in both degrees, as is to be expected from the general setup they consider.

Our goal for the paper at hand is to investigate the situation when $F$ has diagonal structure but no such assumption is made on the shape of $G$. In this instance, one would hope to be able to salvage part of the diagonal structure of $F$ and thus replace the exponential dependence on $k$ in \eqref{1.3} by a polynomial bound. We are able to accomplish this, and indeed obtain bounds that exhibit quadratic growth in $k$.

\begin{theorem}\label{T1.1}
	Suppose that $F, G \in \Z[x_1, \dots, x_s]$ are non-singular forms, where $F$ is diagonal of degree $k$ and $G$ has degree $d \ge 2$. We have
	\begin{align}\label{1.4}
		\frs_0(d;k) \le
		\begin{cases}
			2^k(d+1) & \text{if } d+1 \le k \le d+ 4,\\
			2^d(26+32d) &\text{if } k = d+5,\\
			2^d [(2d+1)k^2 - L_d(k)]  & \text{if }k \ge d+6,
		\end{cases}
	\end{align}
	where $L_d(k)=(4d^2+8d+1)k-2d^3-7d^2-5d-4d\lfloor\sqrt{2k-2d}\rfloor-2\lfloor\sqrt{2k-2d+2}\rfloor$.
\end{theorem}

Our new bounds neatly beat the old bound \eqref{1.3} in all cases. In particular, given that for small values of $k-d$ both \cite{BHB:17} and the proof of our Theorem~\ref{T1.1} rely exclusively on arguments which are based on Weyl differencing, 
it is not clear \emph{a priori} that the diagonal structure of $F$ can be exploited in any meaningful way. However, our results surpass those of Browning and Heath-Brown even in the case when $k=d+1$. Indeed, we have the following. 
\begin{theorem}\label{T1.2}
	Let $Q, C \in \Z[x_1, \ldots, x_s]$ be a pair of non-singular forms, where $Q$ is quadratic and $C$ is of the shape $C(\bx) = c_1 x_1^3 + \ldots + c_s x_s^3$. When $s \ge 25$, the number $N_{C,Q}(X)$ of points $\bx \in \Z^s \cap [-X,X]^s$ satisfying $Q(\bx) = C(\bx)=0$ is given by 
	\begin{align*}
		N_{C,Q}(X) = X^{s-5}(\calC_{C, Q} + O(X^{-\nu})),
	\end{align*}
	where $\nu>0$ and the factor $\calC_{C,Q}$ is non-negative. Moreover, we have $\calC_{C, Q}>0$ whenever 
the matrix underlying $Q$ has at least two positive and two negative eigenvalues. 
\end{theorem}
Thus, our bound $\frs_0(2;3) \le 24$ not only beats the value $\frs_0(2;3) \le 36$ stemming from \eqref{1.3}, but improves even on the result $\frs_0(2;3) \le 28$ obtained by Browning, Dietmann and Heath-Brown \cite[Theorem~1.3]{BDHB:15} for systems of non-singular cubic and quadratic equations by more specialised methods than what we are using in the work at hand.

We can interpret the system \eqref{1.2} as describing a diagonal form, the variables of which are required to lie in a hypersurface of lower degree. It follows from the argument of Wooley \cite[Corollary~14.7]{W:19nec} (refining the conclusion of \cite[Theorem~4.1]{W:12war}) that the number of integral points on a diagonal hypersurface of degree $k$ satisfies the expected asymptotic formula \eqref{1.1} whenever
\begin{align}\label{1.5}
	s \ge k^2-k+2\lfloor \sqrt{2k+2} \rfloor - 1.
\end{align}
Our result in Theorem~\ref{T1.1} shows that if we consider the same problem restricted to a hypersurface of fixed low degree $d$, this bound deteriorates only by a constant factor, depending on $d$. This indicates that we are able to extract the diagonal structure of $F$ essentially as efficiently as can be hoped even in the best scenarios. For instance, for fixed small values of $d$ we obtain the bounds 
	\begin{align*}
		\frs_0(2;k) &\le 20k^2-132k+O(\sqrt k)
		\qquad (k\ge 8),\\
		\frs_0(3;k) &\le 56k^2-488k+O(\sqrt k)
		 \qquad (k\ge 9),\\
		\frs_0(4;k) &\le 144k^2-1552k+O(\sqrt k)
		\qquad (k\ge 10).
	\end{align*}

Theorem~\ref{T1.1} 
can be generalised to the setting where we restrict the diagonal form $F$ to a complete intersection of forms of degree $d$. In other words, we are counting integer solutions to the system of simultaneous equations
\begin{align}\label{1.6}
	F(\bx) = G_1(\bx) = \ldots = G_\rho(\bx) = 0
\end{align}
with $|\bx| \le X$, where again $F$ is a diagonal form and we impose no further constraints on the shapes of the forms $G_1, \dots, G_\rho$. Since we are dealing with complete intersections, however, we will need a notion of non-singularity for the variety defined by $G_1, \dots, G_\rho$.
When $G_1, \dots, G_\rho \in \Z[x_1, \dots, x_s]$, we set
\begin{align*}
	\calV^*(\bG) = \left\{ \bx \in \C^s: \rk \left(\frac{\partial G_i(\bx)}{\partial x_j}  \right)_{\substack{1 \le i \le \rho \\ 1 \le j \le s}} \le \rho-1 \right\},
\end{align*}
noting that if all forms $G_1, \dots, G_\rho$ are non-singular one has the bound
\begin{align*}
	\dim \calV^*(\bG) \le \rho-1.
\end{align*}	

The most general version of our theorem is then as follows.

\begin{theorem}\label{T1.A}
	Let $k>d \ge 2$. For $j \ge 1$, define the functions
	\begin{align*}
		s_0(j) &= \min\{ 2^{j-1},  \textstyle{\frac12} j (j-1) + \lfloor \sqrt{2j+2}\rfloor \} \quad 
\text{ and } \\
		\sig_0(j)& = \min\{2^{j-1}, (j-1) (j-2) + 2 \lfloor\sqrt{2j}\rfloor \}. 
	\end{align*}
	Suppose that the conditions
	\begin{align*}
		s > 2^{d+1} s_0(k-d) + 2^d (\rho+1)d \sig_0(k-d),
	\end{align*}
	\begin{align*}
		\frac{2^{d-1}\rho d}{s-\dim \calV^*(\bG)} + \frac{{2^d}(\rho d+2)\sig_0(k-d)}{s} & < 1,
	\end{align*}
	and		
	\begin{align*}
		\frac{ 2^{d-1} \rho (\rho+1) (d-1) } {s-\dim \calV^*(\bG)} + \frac{{2^d}(\rho+2)\sig_0(k-d)}{s} & < 1 \\
	\end{align*}
	are all satisfied. Then for some $\nu>0$ we have
	\begin{align}\label{1.7}
		N_{F, \bG}(X) = X^{s-k-\rho d}\left(\calC_{F, \bG} + O(X^{-\nu})\right),
	\end{align}
	where $\calC_{F,\bG}$ is a product of local solution densities associated with the system \eqref{1.6}.
\end{theorem}

As will transpire from Theorem~\ref{T2.1} below, the quantities $s_0(j)$ and $\sig_0(j)$ correspond to the best known bounds for Hua's lemma and the best inverse Weyl exponents for exponential sums of polynomials of degree $j$, respectively. 

We explore the strength and reach of Theorem~\ref{T1.A} by analysing its consequences in some of the most relevant special cases. For simplicity, write $\frs_0^*(d, \rho;k)$ for the least number having the property that any system of one non-singular diagonal form of degree $k$ and $\rho$ forms of degree $d$ satisfies an asymptotic formula as in \eqref{1.7} whenever $s-\dim \calV^*(\bG) > \frs_0^*(d,\rho;k)$. Our first interesting special case is that where $\rho=2$, corresponding to diagonal forms that are restricted to the complete intersection of just two hypersurfaces of equal degree.

\begin{corollary}\label{C1.5}
	 We have
	\begin{align*}
		\frs_0^*(d,2;k) \le
		\begin{cases}
			2^{k-1}(2+3d) & \text{if }d+1 \le k \le d+4,\\
			2^d (26+48d) & \text{if }k=d+5,\\
			2^d [(3d+1) k^2  -  L_{d}^*(k)  ]  & \text{if }k \ge d+6,
		\end{cases}
	\end{align*}
	where $L_{d}^*(k)=(6d^2+11d+1)k-3d^3-10d^2-7d-6d\lfloor\sqrt{2k-2d}\rfloor-2\lfloor \sqrt{2k-2d+2}\rfloor $.	
\end{corollary}
For comparison, Theorem~1.2 of Browning and Heath-Brown \cite{BHB:17} produces the bound
\begin{align*}
	\frs_0^*(d,2;k) \le 2^k(k-1)(d+1) + 2^d d.
\end{align*}
As in the case of Theorem~\ref{T1.1}, we note that we are able to replace the exponential dependence on $k$ by an expression that exhibits only quadratic growth in $k$. Correspondingly, our bounds are superior to those that can be obtained in the more general situation in all instances. In the special case of systems of one cubic and two general quadratic equations, we obtain the bound $\frs_0^*(2,2;3) \le 32$, in comparison to the value $\frs_0^*(2,2;3) \le 56$ stemming from the work of \cite{BHB:17}. 

The strategy and methods utilised in our proof of Theorem~\ref{T1.A} and its corollaries continue to be applicable in the case when $F$ has only partial diagonal structure. In particular, we are interested in the situation where $F$ decomposes into a sum of $n$-ary forms in disjoint sets of variables. 
We can now count integer zeros of such forms $F$ that are constrained to lie in a smooth hypersurface of smaller degree. Suppose that $F\in \Z[x_1, \dots, x_s]$ is a form of degree $k$ which decomposes as a sum of $n$-ary forms, and $G \in \Z[x_1, \dots, x_s]$ is non-singular of degree $d$. We write $\frs_0(d; k, n)$ for the smallest integer $\frs_0$ with the property that $N_{F, G}(X)$ satisfies an asymptotic formula of the shape \eqref{1.1} whenever $s > \frs_0$ and $n|s$.
\begin{theorem}\label{T1.7}
	Let $k > d \ge 2$ and $n \ge 1$ be integers, and suppose that $k \ge 6$. Then
	\begin{align*}
		\frs_0(d; k, n) \le 2^{d+1}(d+2)n\big((k-1)(k-2)+2 \lfloor \sqrt{2k}\rfloor\big)  + 2^{d-1}d.
	\end{align*}
\end{theorem}
Observe that this bound is still quadratic in $k$, indicating that when $k$ is sufficiently large in comparison to $d$ and $n$, the effect of the partial diagonal structure of $F$ will be strong enough that the bound of Theorem~\ref{T1.7} prevails over the corresponding bound \eqref{1.3} for the general setting. Unfortunately, the bounds here, when specialised to $n=1$, are somewhat larger than the ones appearing in Theorems \ref{T1.1} and \ref{T1.A}. This is a relic of the more complex setting for sums of $n$-ary forms, which prevents us from making certain technical simplifications in the argument. In order to better assess the strength of our result, it is helpful again to consider one of the most interesting special cases, wherein $F$ is a sum of binary forms and $G$ is a quadratic form of general shape. 
\begin{corollary}\label{C1.8}
	For $k \ge 6$ we have
	\begin{align*}
		\frs_0(2; k, 2) \le 64k^2-192k+ 128\lfloor \sqrt{2k}\rfloor+132.
	\end{align*}
\end{corollary}
One can check that this supersedes the bound of \eqref{1.3} when $k \ge 8$.

As in the diagonal case, Theorem~\ref{T1.7} and Corollary~\ref{C1.8} are special instances of a more general result.
\begin{theorem}\label{T1.B}
	Let $k > d \ge 2$. Suppose that $F\in \Z[x_1, \dots, x_s]$ is a form of degree $k$ which decomposes as a sum of $n$-ary forms, and $G_1, \dots, G_\rho \in \Z[x_1, \dots, x_s]$ are forms of degree $d$. Assume that $n | s$ and let $\sig_0(k)$ be as in Theorem $\ref{T1.A}$.
	Furthermore, suppose that the conditions
	\begin{align*}		
		\frac{2^{d-1}\rho d}{s-\dim \calV^*(\bG)} + \frac{ 2^{d+1}(\rho d+2) n \sig_0(k)}{s} & < 1
	\end{align*}
	and
	\begin{align*}
		\frac{ 2^{d-1} \rho (\rho+1) (d-1) } {s-\dim \calV^*(\bG)} + \frac{2^{d+1}(\rho+2) n \sig_0(k)} {s} & < 1
	\end{align*}
	are both satisfied. Then for some $\nu>0$ we have an asymptotic formula as in \eqref{1.7}, where again the constant $\calC_{F,\bG}$ is a product of local solution densities connected to the system \eqref{1.6}. 	
\end{theorem}

We restate Theorem~\ref{T1.B} in a way that is easier to parse. For simplicity, write $\frs_0^*(d,\rho; k,n)$ for the least integer $\frs_0$ having the property that the asymptotic formula \eqref{1.7} holds whenever $s - \dim \calV^*(\bG) > \frs_0$ and $n|s$. In this formulation, Theorem~\ref{T1.B} implies the following.
\begin{corollary}\label{C1.10}
	Let $k > d \ge 2$, and write $\sig=\sig_0(k)$, where $\sig_0(k)$ is as in Theorem $\ref{T1.A}$.
	We have
	\begin{align*}
		\frs_0^*(d, \rho; k, n) \le \begin{cases}
			\begin{displaystyle}2^{d-1} \left((4n\sig+1)d \rho + 8n\sig  \right) \end{displaystyle} &  (\rho \le 4n\sig),\\
			\begin{displaystyle}2^{d-1}\left((d-1) \rho^2 + (d-1+4n\sig)\rho  + 8n\sig \right)\end{displaystyle}  & (\rho \ge 4 n\sig+1).
		\end{cases}
	\end{align*}
\end{corollary}

The proofs of Theorems \ref{T1.A} and \ref{T1.B} combine ideas of Browning and Heath-Brown \cite{BHB:17} with estimates related to Vinogradov's mean value theorem. We briefly sketch the main idea
of our argument by considering a special case of Theorem~\ref{1.1}. Let $F$ be a diagonal form of degree $k \ge 3$ and $G$ any non-singular quadratic form, so that
\begin{align*}
F(\bx) = \sum_{i=1}^s a_i x_i^k \quad \text{ and }\quad G(\bx) = \bx \cdot B \bx
\end{align*}
for some non-vanishing integer coefficients $a_1, \dots, a_s$ and some non-singular symmetric matrix $B$ with integer and half-integer entries. We perform three Weyl differencing steps on the exponential sum
\begin{align*}
	T(\alp, \bet) = \sum_{|\bx | \le X} e(\alp F(\bx) + \bet G(\bx)).
\end{align*}
This generates three sets of differencing variables $\bg$, $\bh$ and $\bj$, but since each differencing step reduces the degree of the polynomial by one, this procedure eliminates any dependence on $G$. Meanwhile, the polynomial $F$ retains its diagonal structure and transforms into a polynomial
\begin{align*}
	F_3(\bx; \bg, \bh, \bj) = \sum_{i=1}^s a_i p(x_i; g_i, h_i, j_i),
\end{align*}
where $p$ is a quaternary form of total degree $k$ and of degree $k-3$ in the first variable. Thus, the exponential sum $T(\alp, \bet)$ may be bounded in terms of an exponential sum $T_3(\alp) = \prod_{i=1}^s f_3(a_i \alp)$, where
\begin{align*}
	f_3(\alp) = \sum_{g}\sum_{h} \sum_{j} \sum_{x} e(\alp p(x; g, h, j))
\end{align*}
and the summations run over suitable subintervals of $[-2X,2X]$. We have thus eliminated all dependence on $G$ and are left with an exponential sum associated to sums of quaternary forms of degree $k$. Such sums can be bounded in terms of mean values of Vinogradov type of degree $k$. Thus, instead of performing $k-1$ differencing steps as is necessary when $F$ is of general shape, we have to difference only three times, and can then feed the resulting bounds into the technology developed by Browning and Heath-Brown \cite{BHB:17}. This is the key that allows us to essentially replace a factor of $2^k$ by a quadratic polynomial in $k$ in the context of Theorem~\ref{T1.B}. The refinement that produces the stronger Theorem~\ref{T1.A} consists of a more careful treatment in which we difference only twice and average over the corresponding differencing variables $g$ and $h$, now treated as constants. 

\textbf{Plan of the paper.} 
In the next section we collect the mean value estimates and Weyl type bounds we will be requiring. Most of the key results can be extracted from the papers by Bourgain, Demeter and Guth \cite{BDG:16} and Wooley \cite{W:19nec}. In the third and fourth sections, we construct the circle method framework within which the bounds of Section \ref{S2} will be applied. This part of the argument employs ideas developed by Browning and Heath-Brown in their work on forms in many variables with differing degrees \cite{BHB:17}. Finally, in Section \ref{S5} we wrap up the proofs of our main conclusions, Theorems \ref{T1.A} and \ref{T1.B}, and briefly discuss the local solubility for Theorem \ref{T1.2}.
	
\textbf{Notation.} Throughout the paper the following notational conventions will be observed. When $H : [0,1]^n \to \C$ is integrable, we write
\begin{align*}
	\oint H(\balp) \d \balp = \int_{[0,1]^n} H(\balp) \d \balp.
\end{align*}
Also, statements containing the letter $\eps$ are asserted to hold for all sufficiently small values of $\eps$, and we make no effort to track the precise `value' of $\eps$, which is consequently allowed to change from one line to the next. We will be liberal in our use of vector notation. In particular, equations and inequalities involving vectors should always be understood entrywise. In this spirit, we write $| \bx | = \| \bx \|_\infty = \max |x_i|$, as well as $(\ba, b) = \gcd(a_1, \dots, a_n, b)$. For $\alp \in \R$ we write $\| \alp \| = \min_{z \in \Z}|\alp - z|$. Finally, the letter $X$ always denotes a large positive number, and the implicit constants in the Landau and Vinogradov notations are allowed to depend on all parameters but $X$.

\textbf{Acknowledgements.}
The authors are grateful to Tim Browning for asking the question that led to the genesis of this paper.
The bulk of the work was done during two visits of the first author at West Chester University, whose hospitality is also gratefully acknowledged. The first author was supported by Starting Grant 2017-05110 by the Swedish Research Council (Vetenskapsr{\aa}det).

\section{Mean Values and Weyl Estimates}\label{S2}

In this section we collect the mean value estimates and Weyl-type bounds that will be of relevance in our subsequent analysis. When $\varphi \in \Z[x_1, \dots, x_m]$ is a polynomial of degree $j$ and $\calB \subseteq [-X,X]^m$ a convex domain, we set
\begin{align*}
	f_{j,m}(\alp) = f_{j,m}(\alp; \calB) = \sum_{\bx \in \calB \cap \Z^m} e(\alpha \vphi(\bx)).
\end{align*}
The critical results are summarised in the following theorem. 
\begin{theorem}\label{T2.1}
	Let $\vphi \in \Z[x_1, \ldots, x_m]$ be of degree $j \ge 1$, and set
	\begin{align*}
		s_0(j) = \min\{ 2^{j-1},  \textstyle{\frac12} j (j-1) + \lfloor \sqrt{2j+2}\rfloor \} \quad \text{ and } \quad
		\sig_0(j+1) =2 s_0(j) 
	\end{align*}
	with $\sig_0(1)=1$. 
	Let $\calB$ be the image of a unimodular linear transformation of $[-X,X]^m$ having the property that $\max \{|\bx|: \bx \in \calB\} \ll X$.
	We have the following bounds.
	\begin{enumerate}[(a)]
		\item \label{T2.1a}
			Suppose that $u \ge s_0(j)$. Then
			\begin{align*}
				\oint |f_{j,m}(\alp)|^{2u} \d \alp \ll X^{2mu - j + \eps}.
			\end{align*}
		\item \label{T2.1b}
			Let $\alp \in [0,1)$ and $q \in \N$ satisfy $\| \alp q \| \le q^{-1}$. Then
			\begin{align*}
				|f_{j, m}(\alp)| \ll X^{m+\eps} \left( \frac 1 q + \frac 1 X + \frac q {X^j}\right)^{\textstyle{\frac{1}{\sig_0(j)}}}.
			\end{align*}
	\end{enumerate}
\end{theorem}

\begin{proof}
	For the proof of part \eqref{T2.1a} we use an idea presented in the proof of Theorem 8.1 in \cite{W:99}. Suppose first that $\max \{|\bx|: \bx \in \calB\} \le \tau X$ for some $\tau$. By making the change of variables $x_i = x_1+y_i$ for $2 \le i \le m$ and putting 
	$$
		\tilde \vphi(x_1; \by)=\vphi(x_1, x_1+y_2, \ldots, x_1+y_m)= \vphi(\bx), 
	$$
	we discern upon considering the underlying equations that 
	\begin{align}\label{2.1}
		\oint |f_{j,m}(\alp)|^{2u} \d \alp \le \oint |\tilde f_{j,m}(\alp)|^{2u} \d \alp,
	\end{align}
	where 
	\begin{align*}
		\tilde f_{j,m}(\alp) = \sum_{\substack{|y_i| \le 2\tau X \\ 2 \le i \le m}} \sum_{|x_1| \le \tau X} e(\alp \tilde \vphi(x_1; \by)).
	\end{align*}
	It then follows that 
	\begin{align}\label{2.2}
		\oint |\tilde f_{j,m}(\alp)|^{2u} \d \alp 
		& \ll X^{2u(m-1)} \max_{\substack{|y_i| \le 2\tau X \\ 2 \le i \le m}} \oint \Big|\sum_{|x_1| \le \tau X} e(\alp \tilde \vphi(x_1; \by))\Big|^{2u} \d \alp.
	\end{align}	
	The argument of the exponential sum can be viewed as a polynomial of degree $j$ in $x_1$. Furthermore, its leading coefficient is independent of $y_2, \ldots, y_m$, and all other coefficients are bounded by a power of $X$. When $s_0(j)=2^{j-1}$, we thus discern from the proof of Hua's Lemma~\cite{Hua:38} that whenever $u \ge s_0(j)$ we have
	\begin{align*}
		\oint \Big| \sum_{|x_1| \le \tau X}e(\alp \tilde \vphi(x_1; \by)) \Big|^{2u} \d \alp \ll X^{2u-j+\eps}
	\end{align*}
	uniformly in $\by$. The analogous conclusion is reached for larger values of $j$ by reference to Corollary~14.8 of \cite{W:19nec}. In combination with \eqref{2.1} and \eqref{2.2} this proves part \eqref{T2.1a}.
	
	
	In order to prove part \eqref{T2.1b}, we modify the argument of Theorem~5.2 in \cite{V:HL} (see also Theorem~5.1 in \cite{P:05}). Note first that for $\sigma_0(j)=2^{j-1}$ this is the classical version of Weyl's inequality \cite[Lemma 2.4]{V:HL}. We may thus assume that $j\ge 2$.
	Suppose that 
	$
		\calB  \subseteq  [-(\tau-1) X,(\tau-1) X] \times [-\tau X,\tau X]^{m-1} 
	$
	for some constant $\tau>1$. For any algebraically independent set of polynomials $\bPsi = \{\Psi_1, \dots, \Psi_r\} \subseteq \Z[x_1, \dots, x_m]$ we put
	\begin{align*}
		f_{\bPsi}(\balp; \calB) = \sum_{\bx \in \calB \cap \Z^m} e \left( \sum_{i=1}^r \alp_i \Psi_i(\bx)\right).
	\end{align*}
	
	Let $\langle \vphi \rangle =\{\vphi_1, \ldots, \vphi_r\}$ be a basis of the vector space generated by $\vphi$ and all its partial derivatives. Clearly, we may choose this basis in such a way that $\vphi_1=\vphi$ and $\vphi_2 = \partial_{x_1}\vphi$. 	
	Consider a set $\calZ \subseteq [-X,X] \cap \Z$ with $|\calZ|=Z$. For any $z \in \calZ$ and $\bx \in \calB$ we have $|\bx + (z,0,\ldots,0) | \le \tau X$. Thus, from Lemma 5.2 of \cite{B:14} we discern that
	\begin{align*}
		f_{\langle \vphi \rangle}(\balp; \calB) &= \sum_{(x_1-z, x_2, \dots, x_m) \in \calB \cap \Z^m} e \left( \sum_{i=1}^r \alp_{i} \vphi_{i}(x_1-z, x_2, \dots, x_m)\right)\\
		&\ll (\log X)^m\sup_{\bbet \in [0,1)^m}|g(z;\bbet)|,
	\end{align*}
	where
	\begin{align*}
		g(z; \bbet)= \sum_{|\bx| \le \tau X} e \left( \sum_{i=1}^r \alp_{i} \vphi_{i}(x_1-z, x_2, \dots, x_m) - \bbet \cdot \bx\right).
	\end{align*}
	Here, we used that the proof of Lemma 5.2 in \cite{B:14} remains unchanged if the linear transformation is combined with a shift of controlled size. 
	Upon averaging over all possible values of $z \in \calZ$ and applying H\"older's inequality, it follows that for any $u \in \N$ one has
	\begin{align}\label{2.3}
		|f_{\langle \vphi \rangle}(\balp; \calB)|^{2u} \ll Z^{-1} (\log X)^{2um} \sup_{\bbet \in [0,1)^m}\sum_{z \in \calZ}|g(z; \bbet)|^{2u}.
	\end{align}
	
	The goal is now to estimate the inner sum by means of the large sieve. We begin by observing that, upon applying the binomial theorem and re-arranging the order of summation, one has
	\begin{align*}
		\sum_{i=1}^r \alp_{i} \vphi_{i}(x_1-z, x_2, \dots, x_m) = \sum_{i=1}^r \gam_{i}(z)\vphi_i(x_1, \dots, x_m)
	\end{align*}
	for suitable coefficients $\gam_{i}(z)$. In particular, we have
	\begin{align}\label{2.4}
		\gam_{2}(z) = -\alp_{1}z + \alp_2.
	\end{align}
	Set 
	\begin{align*}
		A = \max_{|\bx| \le \tau X}uX^{1-j}|\vphi_2(\bx)| .
	\end{align*}
	Denote further by $\calC (n)$ the set of integer points $\bx_1, \dots, \bx_u \in [-\tau X, \tau X]^m$ that satisfy
	\begin{align*}
		\vphi_{2}(\bx_1) + \ldots + \vphi_{2}(\bx_u) = n,
	\end{align*}
	and set
	\begin{align*}
		a(n, \bbet)= \sum_{(\bx_1, \dots, \bx_u) \in \calC(n)}  e \left( \sum_{\substack{i=1 \\ i \neq 2}}^r\gam_{i}(z)(\vphi_{i}(\bx_1)+ \ldots +\vphi_{i}(\bx_u))   - \bbet \cdot(\bx_{1}+\ldots + \bx_{u}) \right).
	\end{align*}
	Clearly, $\calC(n)$ is empty unless $|n| \le A X^{j-1}$, so we find that
	\begin{align*}
		\sum_{z \in \calZ}|g(z; \bbet)|^{2u} \le \sum_{z \in \calZ} \left|\sum_{|n| \le A X^{j-1}} a(n,\bbet) e\left( \gam_2(z) n\right)\right|^2.
	\end{align*}
	
	Suppose now that $\del$ is a positive real number having the property that for any two elements $z_1, z_2 \in \calZ$ with $z_1 \neq z_2$ one has
	\begin{align}\label{2.5}
		\| \gam_{2}(z_1)-\gam_{2}(z_2)\| > \del.
	\end{align}
	It then follows from the large sieve inequality (see {\it e.g.} \cite[Theorem 9.1]{FI:OdC}) that
	\begin{align}\label{2.6}
		\sum_{z \in \calZ}|g(z; \bbet)|^{2u} \ll \left(X^{j-1} + \del^{-1}\right) \sum_{|n| \le A X^{j-1}}|a(n,\bbet)|^2.
	\end{align}
	From applying the triangle inequality and referring to part \eqref{T2.1a} of the theorem with $\vphi_2$ in the place of $\vphi$, it transpires that whenever  $u \ge s_0(j-1)=\frac12 \sig_0(j)$, one has the bound
	\begin{align}\label{2.7}
		\sum_{|n| \le A X^{j-1}}|a(n, \bbet)|^2 &\le \sum_{|n| \le A X^{j-1}}\left|\sum_{(\bx_1, \dots, \bx_u) \in \calC(n)}1\right|^2 \notag \\ &=\oint |f_{j-1,m}(\alp; [-\tau X,\tau X]^m)|^{2u} \d \alp \ll  X^{2um-(j-1) + \eps}.
	\end{align}
	Thus, if we can choose the set $\calZ$ in such a way that in the spacing condition \eqref{2.5} we may take $\del\gg X^{1-j}$, the bounds \eqref{2.3}, \eqref{2.6} and \eqref{2.7} imply that
	\begin{align}\label{2.8}
		|f_{\langle \vphi \rangle}(\balp; \calB)|^{2u} \ll X^{2um+\eps} Z^{-1}.
	\end{align}
	The challenge is therefore to choose $\calZ$ as large as possible with the required properties.
	
	From \eqref{2.4} we discern that
	\begin{align*}
		\| \gamma_2(z_1) - \gamma_2(z_2) \|=\| \alpha_1 (z_2-z_1) \|.
	\end{align*}
	Suppose that
	\begin{align}\label{2.9}
		|\alp_1q-a| \le q^{-1}
	\end{align}
	for some $a \in \Z$ and $q \in \N$, and set $P=\min\{q,X\}$. For a fixed $z_1 \in [-P,P]$ we estimate the number of choices for $z_2 \in [-P,P]$ that satisfy
	\begin{align}\label{2.10}
		\| \alp_1(z_2-z_1)\| \le X^{1-j}.
	\end{align}
	Suppose that $z_2$ satisfies \eqref{2.10}. It follows from \eqref{2.9} by the triangle inequality that
	\begin{align*}
		\|(a/q)(z_2-z_1)\| \le X^{1-j} + P/q^2.
	\end{align*}
	Since $P \le q$, we infer that the points $z_2 \in [-P,P]$ satisfying \eqref{2.10} for any given $z_1$ lie in a set $\calR(z_1)$ of cardinality at most
	\begin{align*}
		R \ll qX^{1-j}+ q^{-1}P + 1.
	\end{align*}
	We choose the set $\calZ$ in such a way that it contains at most one element of each of the sets $\calR(z_1)$ as $z_1$ ranges over the interval $[-P,P]$. This set is of size $Z \ge P/(R+1)$, so that
	\begin{align*}
		Z^{-1}\ll \frac{1}{P} + \frac{1}{q} + \frac{Xq}{PX^j} \ll \frac{1}{X} + \frac{1}{q} + \frac{q}{X^j}.
	\end{align*}
	Thus, on recalling that $2u \ge \sig_0(j)$, we obtain from \eqref{2.8} the bound
	\begin{align}\label{2.11}
		|f_{\langle \vphi \rangle}(\balp; \calB)| \ll X^{m+\eps} \left(\frac{1}{X} + \frac{1}{q} + \frac{q}{X^{j}}\right)^{{\textstyle{\frac{1}{\sig_0(j)}}}}
	\end{align}
	with $\sig_0(j)$  as in the statement of the theorem.
	The claim of part \eqref{T2.1b} of the theorem follows from \eqref{2.11} upon specialising to $\alp_2 = \dots =\alp_r = 0$. 
\end{proof}

Whilst the minor arcs bound of Theorem \ref{T2.1}\eqref{T2.1b} is already sufficient for our purposes in Theorem \ref{T1.B}, for the result of Theorem \ref{T1.A} we need a more refined argument. For a parameter $Q \le X$ we denote by $\frM_j(Q)$ the set of $\alp \in [0,1)$ having the property that $\| \alp q \| \le QX^{-j}$ for some natural number $q \le Q$, and set $\frm_j(Q) = [0,1) \setminus \frM_j(Q)$. The bound in Theorem~\ref{T2.1}\eqref{T2.1b} in the case $m=1$ then implies that
\begin{align*}
	\sup_{\alp \in \frm_j(Q)} |f_{j,1}(\alp)| \ll X^{1+\eps} Q^{-1/\sig_0(j)}.
\end{align*}
In our analysis, we will be led to consider exponential sums of degree $j$ on a set of minor arcs associated with the degree $j+r$ for certain $r \ge 1$. The treatment of such scenarios will be greatly facilitated by the following lemma.

\begin{lemma}\label{L2.2}
	Let $j, r \in \N$ and suppose that $\alp \in \frm_{j+r}(Q)$ for some $Q \le X$.
	Then we have
	\begin{align*}
		\sum_{h=1}^{X^r} |f_{j,1}(h \alp)| \ll X^{r+1+\eps}Q^{-\textstyle{\frac{1}{\sig_0(j)}}}.
	\end{align*}
\end{lemma}	

\begin{proof}
	By Dirichlet's Theorem, we obtain integers $a$ and $q$ with $(a,q)=1$ satisfying
	\begin{align}\label{2.12}
		|\alpha - a/q| \le \frac{Q}{qX^{j+r}} \quad \mbox{and} \quad 1 \le q \le \dfrac{X^{j+r}}{Q}.
	\end{align}
	Suppose first that $j=1$. In such a situation, the exponential sum $f_{1,1}(\alp)$ is a linear sum, and we have the familiar bound $|f_{1,1}(\alp)| \ll \min\{X, \|\alp\|^{-1}\}$.
		It thus follows from Lemma~2.2 of \cite{V:HL} that
	\begin{align*}
		\sum_{h=1}^{X^r} |f_{j,1}(h \alp)| &\ll \sum_{h=1}^{X^r} \min\{X, \|h\alp\|^{-1}\} 
		\ll X^{r+1+\eps}\left(\frac{1}{q} + \frac{1}{X}+\frac{q}{X^{r+1}}\right).
	\end{align*}	
	Since $\alpha \in \frm_{1+r}(Q)$, we have $q > Q$, and the desired conclusion follows easily for $j=1$.
	
	Now suppose that $j \ge 2$. To economise on clutter, we temporarily abbreviate $\sig_0(j)$ to $\sig$. Clearly, we have
	\begin{align}\label{2.13}
		\sum_{h=1}^{X^r} |f_{j,1}(h \alp)| &= 	\sum_{\substack{h=1\\ h \alp \in \frm_{j}(Q)}}^{X^r} |f_{j,1}(h \alp)| + \sum_{\substack{h=1\\ h \alp \in \frM_{j}(Q)}}^{X^r} |f_{j,1}(h \alp)|,
	\end{align}
	and upon invoking the conclusion of Theorem~\ref{T2.1}\eqref{T2.1b} and summing trivially we discern that
	\begin{align}\label{2.14}
		\sum_{\substack{h=1\\ h \alp \in \frm_{j}(Q)}}^{X^r} |f_{j,1}(h \alp)| \ll X^{r+1+\eps} Q^{-1/\sig}.
	\end{align}
	
	Assume now that $h \alp \in \frM_j(Q)$, so that there exists a positive integer $q_h \le Q$ for which $\|q_h h \alp\| \le QX^{-j}$. Note that for $Q \le X$ and $j \ge 2$ we have  $QX^{-j} \le X^{-1}$. Thus by a standard transference principle \cite[Lemma~A.1]{W:15sae5} we find that
	\begin{align*}
		 |f_{j,1}(h \alp)| \ll X^{1+\eps}
		 \left(\frac{1}{q_h+X^j\|q_h h \alp\|} + \frac{1}{X}+ \frac{q_h+X^j\|q_h h \alp\|}{X^{j}}\right)^{1/\sig},
	\end{align*}
	and one easily shows that the first term dominates. We thus infer that
	\begin{align}\label{2.15}
		 \sum_{\substack{h=1\\ h \alp \in \frM_{j}(Q)}}^{X^r} |f_{j,1}(h \alp)| &\ll  X^{1+\eps}\sum_{h=1}^{X^r} (\min\{q_h^{-1}, X^{-j} \|q_h h \alp\|^{-1}\})^{1/\sig} \nonumber\\
		 &\ll X^{1+\eps -j/\sig}  \sum_{h=1}^{X^r} \left( \sum_{q_h \le Q}\min\{X^j q_h^{-1},\|q_h h \alp\|^{-1}\}\right)^{1/\sig}.
	\end{align}
	By H\"older's inequality and a divisor estimate, the sum can be bounded as
	\begin{align}\label{2.16}
		 &\sum_{h=1}^{X^r} \left( \sum_{q_h \le Q} \min\{X^j q_h^{-1},\|q_h h \alp\|^{-1}\}\right)^{1/\sig} \nonumber
		 \\&\qquad\ll X^{r(1-1/\sig)} \left( \sum_{h \le X^r}\sum_{q_h \le Q}\min\{X^jq_h^{-1},\|q_h h \alp\|^{-1}\}\right)^{1/\sig} \nonumber\\
		 &\qquad\ll X^{r(1-1/\sig)} X^{\eps}\left( \sum_{h_1 \le QX^r}\min\{X^{j+r}h_1^{-1},\|h_1 \alp\|^{-1}\}\right)^{1/\sig}.
	\end{align}
	We now deduce from \eqref{2.12} and \cite[Lemma~2.2]{V:HL} that
	\begin{align*}
		\sum_{h_1 \le QX^r}\min\{X^{j+r} h_1^{-1},\|h_1 \alp\|^{-1}\} \ll X^{j+r+\eps}\biggl(\dfrac{1}{q}+\dfrac{Q}{X^j}+\dfrac{q}{X^{j+r}}\biggr) \ll X^{j+r+\eps} Q^{-1},
	\end{align*}
	since $\alp \in \frm_{j+r}(Q)$ implies that $q > Q$.
	The proof is complete upon inserting this estimate into \eqref{2.16} and recalling \eqref{2.13}, \eqref{2.14}, and \eqref{2.15}.
\end{proof}

We will apply Lemma~\ref{L2.2} in the following form.
\begin{lemma}\label{L2.3}
	Suppose that $\alp \in \frm_{j+r}(Q)$ for some $Q \le X$.
	Then we have
	\begin{align*}
		\sum_{\substack{|h_1|, \dots, |h_r| \le X\\ h_1 \cdots h_r \neq 0}} |f_{j,1}(h_1 \cdots h_r \alp)| \ll X^{r+1+\eps}Q^{-\textstyle{\frac{1}{\sig_0(j)}}}.
	\end{align*}
\end{lemma}	
\begin{proof}
	This follows directly from Lemma~\ref{L2.2} via a standard divisor estimate.
\end{proof}

\section{The circle method framework: Minor arcs}\label{S3}
We now present the framework within which the results of the previous section will be applied. From this point onwards we consider forms $F, G_1, \dots, G_\rho \in \Z[x_1, \dots, x_s]$, where $G_1, \dots, G_\rho$ are of degree $d$ and $F$ has degree $k$. Here we always assume that $k > d \ge 2$. Upon writing
$$
	T(\alp,\bbet) = T(\alp, \bbet; X) = \sum_{|\bx| \le X} e\left(\alp F(\bx) + \sum_{i=1}^\rho \bet_i G_i(\bx)\right),
$$
our counting function is given by
$$
	N_{F, \bG}(X) = \oint T(\alp, \bbet) \d \alp \d \bbet.
$$
We collect the input we will be needing from Section \ref{S2}. Let $s, t, t_0 \in \N$ and $\sig, \Del \in \R_{\ge 0}$ be parameters. Also, recall the definition of the major and minor arcs from the previous section. For simplicity, we write $\frM(\tet) = \frM_k(X^\tet)$ and $\frm(\tet) = \frm_k(X^\tet)$.

\begin{hypothesis}
	We say that the parameter tuple $(s, t, t_0, \sig, \Del)$ satisfies the Central Hypothesis if for every $\tet \in (0, 1]$ and every $\bbet \in [0,1)^\rho$ one has
	\begin{align}\tag{H1}\label{H1}
		\sup_{\alp \in \frm(\tet)} |T(\alp, \bbet)| \ll X^{s-(t/\sig) \tet+\eps}
	\end{align}
	and
	\begin{align}\tag{H2}\label{H2}
		\int_{\frm(\tet)} |T(\alp, \bbet)| \d \alp \ll X^{s-k+\Del - {\textstyle{\frac{t-2t_0}{\sig}}}\tet+\eps}.
	\end{align}
\end{hypothesis}
The goal for this section is to provide a proof of the following result.

\begin{proposition}\label{P3.1}
	Assume the Central Hypothesis for a tuple $(s,t,t_0, \sig, \Del)$. Suppose further that these parameters satisfy
	\begin{align} \label{3.1}
		t &> 2t_0 + (\Del + \rho d)\sig
	\end{align}
	as well as
	\begin{align}
		\frac{2^{d-1}\rho d}{s - \dim \calV^*(\bG)} + \frac{(\rho d+2)\sig}{t} &< 1 \label{3.2}	
\end{align}		
and
\begin{align}
		\frac{2^{d-1}\rho (\rho+1)(d-1)}{s - \dim \calV^*(\bG)} + \frac{(\rho+2)\sig}{t} &< 1. \label{3.3}
	\end{align}
	Then for some $\nu > 0$ we have
	\begin{align*}
		N_{F, \bG}(X) = X^{s-k-\rho d} \chi_\infty \prod_p \chi_p + O(X^{s-k-\rho d - \nu}),
	\end{align*}
	where the local factors $\chi_\infty$ and $\chi_p$ encode the local solubility data for the system \eqref{1.6}. In particular, the Euler product converges absolutely, and all factors are positive if the system \eqref{1.6} has a non-singular solution in $\R$ as well as in all fields $\Q_p$.
\end{proposition}

Our strategy for proving Proposition \ref{P3.1} follows the approach of Browning and Heath-Brown \cite{BHB:17} for forms in many variables with differing degrees (see also \cite{B:18} for an exposition that is notationally closer). We start by bounding the contribution to $N_{F, \bG}(X)$ that arises from $\alp \in \frm(\tet)$ for a suitable parameter $\tet$.

\begin{lemma}\label{L3.2}
	Assume the Central Hypothesis for $(s,t,t_0,\sig,\Del)$ as well as \eqref{3.1}.
	Also, for a number $\tet_* \in (0,1]$ suppose that
	\begin{align}\label{3.4}
		(t/\sig -2)\tet_* > \rho d.
	\end{align}
	There exists a positive real number $\nu$ with the property that for all $\tet \in [\tet_*, 1]$ one has
	\begin{align*}
		\oint \int_{\frm(\tet)} |T(\alp, \bbet)| \d \alp \d \bbet \ll X^{s-k-\rho d - \nu}.
	\end{align*}
\end{lemma}

\begin{proof}	
	Consider a sequence $(\tet_i)$ with $1 = \tet_0 > \tet_1 > \ldots > \tet_T = \tet_*$. By \eqref{3.4} we can choose this sequence with $T = O(1)$ and such that
	\begin{align}\label{3.5}
		2(\tet_{i-1} - \tet_i) < (t/\sig-2)\tet_* -  \rho d \qquad (1 \le i \le T).
	\end{align}
	From \eqref{H2} we have via \eqref{3.1} that
	\begin{align*}
		\oint \int_{\frm(\tet_0)} |T(\alp, \bbet)| \d \alp \d \bbet \ll X^{s-k-\rho d - \nu}
	\end{align*}
	for some suitable $\nu > 0$. Note also that the major arcs $\frM(\tet)$ are disjoint for all $\tet \le 1$ whenever $k \ge 2$, and we have $\vol \frM(\tet) \ll X^{-k+2\tet}$. Thus we see from \eqref{H1} that
	\begin{align*}
		\oint \int_{\frm(\tet_i) \setminus \frm(\tet_{i-1})} |T(\alp, \bbet)| \d \alp \d \bbet & \ll \vol \frM(\tet_{i-1}) \sup_{\substack{\alp \in \frm(\tet_i)\\\bbet \in [0,1)^\rho}} |T(\alp, \bbet)|\\
		& \ll X^{-k+2 \tet_{i-1}} X^{s - (t/\sig)\tet_i+\eps}.
	\end{align*}
	Since \eqref{3.5} implies that
	\begin{align*}
		2 \tet_{i-1} - (t/\sig) \tet_i = 2(\tet_{i-1} - \tet_i) - (t/\sig-2) \tet_i < - \rho d,
	\end{align*}
	the contribution from $\frm(\tet_i) \setminus \frm(\tet_{i-1})$ is compatible with the claim of the lemma for all $i$. The full statement of the lemma follows upon noting that
	\begin{align*}
		\frm(\tet_T) = \frm(\tet_0) \cup \bigcup_{i=1}^{T} (\frm(\tet_i) \setminus \frm(\tet_{i-1}))
	\end{align*}
	and recalling that $T = O(1)$.
\end{proof}

The next step is to introduce a dissection into major and minor arcs for $\bbet$.
\begin{lemma}\label{L3.3}
	Let $\tet \in (0,1]$, and suppose that $\alp \in \frM(\tet)$, with $q$ denoting the denominator of the associated rational approximation to $\alp$. Further let $\kap$ be a positive real number. For any $\eta$ satisfying
	\begin{align}\label{3.6}
		0 < \eta \le 1-\tet,
	\end{align}
	one of three alternatives is satisfied:
	\begin{enumerate}[(A)]
		\item One has
		\begin{align}\label{3.7}
			|T(\alp, \bbet)| \ll X^{s - \kap \eta + \eps}.
		\end{align}
		\item There exists a natural number $r \le X^{(\rho-1)\eta}$ satisfying $$\|q r \bet_i \| \ll X^{-d + \rho (d-1) \eta+\tet} \qquad (1 \le i \le \rho).$$
		\item One has
		\begin{align*}
			s - \dim \calV^*(\bG) \le 2^{d-1}\kap.
		\end{align*}
	\end{enumerate}
\end{lemma}

\begin{proof}
	This is essentially Lemma~6.1 in \cite{BHB:17} (see also Lemma~2.4 in \cite{B:18}).
\end{proof}

Henceforth we assume that
\begin{align}\label{3.8}
	s - \dim \calV^*(\bG) > 2^{d-1}\kap,
\end{align}
so that the third case in Lemma~\ref{L3.3} can be excluded. Moreover, upon combining \eqref{3.7} with \eqref{H1}, it transpires that no generality is lost if we set
\begin{align}\label{3.9}
	\kap \eta = (t/\sig)\tet.
\end{align}

We now define a $(\rho+1)$-dimensional set of major arcs. For a constant $c$, denote by  $\frN(\eta) = \frN(\eta, \tet)$ the set of $(\alp, \bbet) \in [0,1)^{\rho+1}$ having the property that $|\alp q - a| \le cX^{-k+\tet}$ for some $1 \le a \le q \le c X^{\tet}$, and $| \bet_i q r  - b_i| \le c X^{-d + \rho(d-1)\eta + \tet}$ for some  $r \le c X^{\rho(d-1)\eta}$ and $1 \le b_i \le qr$, where $1 \le i \le \rho$. We then set $\frn(\eta) = [0,1)^{\rho+1} \setminus \frN(\eta)$. Note that the constant $c$ can be chosen in such a way that this dissection reflects the case distinction in Lemma~\ref{L3.3}. In particular, upon combining \eqref{H1} and \eqref{3.7} by means of the relation \eqref{3.9}, we see that $|T(\alp, \bbet)| \ll X^{s-\kap \eta+\eps}$ for all $(\alp, \bbet) \in \frn(\eta)$. Meanwhile, we compute
\begin{align}\label{3.10}
 	\vol \frN(\eta) &\ll \sum_{1 \le q \le cX^\tet} \sum_{1 \le a \le q} \frac{X^{-k+\tet}}{q} \sum_{1 \le r \le cX^{\rho(d-1)\eta}} \left( \sum_{1 \le b \le qr} \frac{X^{-d+\rho(d-1)\eta + \tet}}{qr} \right)^\rho \nonumber\\
 	& \ll X^{-k - \rho d + \rho(\rho+1)(d-1)\eta + (\rho+2)\tet} \nonumber\\
 	& \ll X^{-k- \rho d + {\textstyle{[{\scriptstyle{\rho(\rho+1)(d-1) +}} \frac{(\rho+2)\kap \sig}{t}}]}\eta},
\end{align}
where in the last step we used \eqref{3.9}.

Our second pruning step involves the $(\rho+1)$-dimensional major arcs $\frN(\eta)$.

\begin{lemma}\label{L3.4}
	Suppose that the Central Hypothesis holds for $(s, t, t_0, \sig, \Del)$. Further assume \eqref{3.1} as well as
	\begin{align}\label{3.11}
		\frac{\rho d}{\kap} + \frac{(\rho d + 2)\sig}{t} &< 1
	\end{align}
	and
	\begin{align}\label{3.12}
		\frac{\rho(\rho+1)(d-1)}{\kap} + \frac{(\rho+2)\sig}{t} &< 1.
	\end{align}	
	Then for any $\eta$ satisfying
	\begin{align}\label{3.13}
		0 < \eta \le \left(1+\frac{\sig \kap}{t}\right)^{-1}
	\end{align}
	we have the bound
	\begin{align*}
		\int_{\frn(\eta)} |T(\alp, \bbet)| \d \alp \d \bbet \ll X^{s-k-\rho d - \nu}
	\end{align*}
	for some $\nu > 0$.
\end{lemma}

\begin{proof}
	Let $\tet_*$ be as in Lemma~\ref{L3.2} and assume that \eqref{3.4} holds. In view of \eqref{3.9}, we write $\eta_* = \frac{t}{\kap \sig}\tet_*$. Thus from Lemma~\ref{L3.2} and \eqref{3.7} we have
	\begin{align*}
		\int_{\frn(\eta_*)} |T(\alp, \bbet)| \d \alp \d \bbet & \ll \oint \int_{\frm(\tet_*)}|T(\alp, \bbet)| \d \alp \d \bbet  + \vol \frM(\tet_*) \sup_{\substack{(\alp, \bbet) \in \frn(\eta_*) \\ \alp \in \frM(\tet_*)}} |T(\alp, \bbet)|\\
		& \ll X^{s-k-\rho d-\nu} + X^{-k+2\tet_*} X^{s-\kap \eta_* + \eps},
	\end{align*}
	and by \eqref{3.9} and \eqref{3.4} we see that the exponent in the second term is acceptable.
	
	Consider now a sequence $(\eta_i)$ with $\eta_* = \eta_0 > \eta_1 > \ldots > \eta_T = \eta > 0$ satisfying
	\begin{align}\label{3.14}
		\eta_{i-1}-\eta_{i} < \left(1-\frac{\rho(\rho+1)(d-1)}{\kap}-\frac{(\rho+2)\sig}{t}\right)\eta \qquad (1 \le i \le T).
	\end{align}
	This is possible by \eqref{3.12}, and we may take $T=O(1)$. Note further that our condition \eqref{3.13} ensures via \eqref{3.9} that the hypothesis \eqref{3.6} of Lemma~\ref{L3.3} is satisfied. It then follows from \eqref{3.10} and Lemma~\ref{L3.3} that
	\begin{align*}
		\int_{\frn(\eta_i)\setminus \frn(\eta_{i-1})}|T(\alp, \bbet)| \d \alp \d \bbet
		&\ll \vol \frN(\eta_{i-1}) \sup_{(\alp, \bbet)\in \frn(\eta_i) }  |T(\alp, \bbet)|\\
		&\ll X^{-k-\rho d + {\textstyle [\frac{\rho(\rho+1)(d-1)}{\kap} + \frac{(\rho+2)\sig}{t}]}\kap \eta_{i-1}} X^{s-\kap \eta_i + \eps}\\
		&\ll X^{s-k-\rho d + \kap(\eta_{i-1} - \eta_i) -  \big[1-{ \textstyle\frac{\rho(\rho+1)(d-1)}{\kap} - \frac{(\rho+2)\sig}{t}}\big]\kap \eta + \eps},
	\end{align*}
	and we infer from \eqref{3.14} that the exponent is acceptable. Since
	\begin{align*}
		\frn(\eta)= \frn(\eta_*) \cup \bigcup_{i=1}^T (\frn(\eta_i) \setminus \frn(\eta_{i-1})),
 	\end{align*}
	we have shown that the conclusion of the lemma holds true as soon as there exists a suitable number $\tet_*$ (and thus $\eta_*$) satisfying \eqref{3.4}.
	
	In order to be able to apply Lemma~\ref{L3.3}, we need to ensure that \eqref{3.6} is satisfied at each stage. In view of \eqref{3.9}, we thus require in particular that
	\begin{align*}
		\left( 1 + \frac{t}{\sig \kap}\right)\tet_* \le 1,
	\end{align*}
	and this bound is compatible with \eqref{3.4} whenever
	\begin{align*}
		\rho d \left(1+\frac{t}{\sig \kap}\right) < t/\sig-2,
	\end{align*}
	which can easily be rearranged to \eqref{3.11}.
\end{proof}
	
We now summarise our results up to this point. Under the hypotheses of Proposition \ref{P3.1} we can find a value $\kap$ satisfying \eqref{3.8} as well as the hypotheses \eqref{3.11} and \eqref{3.12} of Lemma~\ref{L3.4}. Under such circumstances, it follows that for any sufficiently small $\eta > 0$ we have the asymptotic formula
\begin{align}\label{3.15}
	N_{F, \bG}(X) = \int_{\frN(\eta)} T(\alp, \bbet) \d \alp \d \bbet + O(X^{s-k-\rho d - \nu})
\end{align}
for some small positive number $\nu$.
	
\section{The circle method framework: Major arcs}\label{S4}
	
In order to understand the contribution from the major arcs $\frN(\eta)$, it is convenient to work over a modified and slightly enlarged set of major arcs instead. Let
\begin{align}\label{4.1}
	\ome = \rho (d-1) \eta + \tet
\end{align}
and denote by $\frP(\ome)$ the set of $(\alp, \bbet) \in [0,1)^{\rho+1}$ having the property that there exists a natural number $q \le c' X^{\ome}$ and an  integer $(\rho+1)$-tuple $(a, \bb)$ satisfying
\begin{align*}
	|\alp - a/q| \le c' X^{-k+\ome} \qquad \text{ and } \qquad |\bet_i - b_i/q| \le c' X^{-d+\ome} \qquad (1 \le i \le \rho)
\end{align*}
for some suitable constant $c'$. We take $c'$ such that $\frN(\eta) \subseteq \frP(\ome)$, so that under the hypotheses of Proposition \ref{P3.1} we conclude from Lemma~\ref{L3.4} that
\begin{align*}
	\int_{\frP(\ome) \setminus \frN(\eta)} T(\alp, \bbet) \d \alp \d \bbet \ll  \int_{\frn(\eta)} |T(\alp, \bbet)| \d \alp \d \bbet  \ll X^{s-k-\rho d - \nu}
\end{align*}
for some $\nu >0$. A straightforward calculation shows further that
\begin{align*}
	\vol \frP(\ome) \ll X^{-k-\rho d + (2 \rho + 3)\ome}.
\end{align*}

We now introduce our generating functions by setting
\begin{align*}
	S(q; a, \bb) = \sum_{\bx=1}^q e\left(q^{-1}\left(a F(\bx) + \sum_{i=1}^\rho b_i G_i(\bx)\right)\right)
\end{align*}
and
\begin{align*}
	v_X(\gam, \bdel) = \int_{[-X,X]^s} e \left(\gam F(\bxi) + \sum_{i=1}^\rho \del_i G_i(\bxi) \right) \d \bxi.
\end{align*}
Thus, when $\alp = a/q + \gam$ and $ \bet_i = b_i/q + \del_i$ for $1 \le i \le \rho$, it follows from standard arguments that
\begin{align}\label{4.2}
	|T(\alp, \bbet) - q^{-s}S(q; a, \bbet)  v_X(\gam, \bdel)| \ll X^{s-1}q(1+X^k |\gam| + X^d |\bdel|),
\end{align}
and we note that the right hand side is $\ll X^{s-1+2\ome}$ whenever $(\alp, \bbet) \in \frP(\ome)$. It follows that the major arcs contribution can be rewritten as
\begin{align}\label{4.3}
	\int_{\frP(\ome)} T(\alp, \bbet) \d \alp \d \bbet &= \sum_{q \le c' X^{\ome}} q^{-s} \sum_{\substack{a, \bb = 1 \\ (q, a, \bb)=1}}^q S(q; a, \bb) \int_{\substack{|\gam| \le c' X^{-k+\ome} \\ |\bdel| \le c' X^{-d + \ome}}} v_X(\gam, \bdel) \d \gam \d \bdel \nonumber \\
	 & \qquad + O(X^{s-k-\rho d - 1 + (2\rho+5)\ome}).
\end{align}
For any large number $T$, define the truncated singular series
\begin{align*}
	\frS(T) = \sum_{1 \le q \le T}q^{-s} \sum_{\substack{a, \bb = 1 \\ (q, a, \bb)=1}}^q S(q; a, \bb)
\end{align*}
and the truncated singular integral
\begin{align*}
	\frJ(T) = \int_{|\gam|, |\bdel| \le T} v_1(\gam, \bdel) \d \gam \d \bdel.
\end{align*}
Noting that
\begin{align}\label{4.4}
	v_X(\gam, \bdel) = X^s v_1(X^k \gam, X^d \bdel),
\end{align}
we discern from \eqref{4.3} after a change of variables that
\begin{align}\label{4.5}
	\int_{\frP(\ome)} T(\alp, \bbet) \d \alp \d \bbet = X^{s-k-\rho d} \frS(c' X^\ome) \frJ(c' X^\ome) + O(X^{s-k-\rho d - 1 + (2 \rho+5)\ome}).
\end{align}
It thus remains to understand the truncated singular series and integral. We first study the singular integral. Here we will assume the hypotheses of Proposition \ref{P3.1} to hold throughout.
\begin{lemma}\label{L4.1}
	For any $(\gam, \bdel) \in \R^{\rho+1}$ we have
	\begin{align*}
		|v_1(\gam, \bdel)| \ll \min \left\{ 1, |\gam|^{-t/\sig + \eps}, |\bdel|^{- \big({\textstyle{\frac{\rho(d-1)}{\kap} + \frac{\sig}{t} }}\big)^{-1} \! + \, \eps} \right\}.
	\end{align*}
\end{lemma}

\begin{proof}
	We may assume without loss of generality that $\max\{|\gam|, |\bdel|\}>1$, the claim being trivially true otherwise. For some suitably large number $A$ to be determined later, put $P=(\max\{|\gam|, |\bdel|\})^A$, and set $\phi=P^{-k}\gam$ and $\bpsi =  P^{-d} \bdel$.
	
	We use \eqref{3.9} in order to rewrite \eqref{4.1} in the shape
	\begin{align}\label{4.6}
		\kap \eta = \left(\frac{\rho(d-1)}{\kap} + \frac{\sig}{t}\right)^{-1} \ome.
	\end{align}
	Determine $\eta$ such that, recalling \eqref{3.9} and \eqref{4.6}, one has
	\begin{align}\label{4.7}
		\max\{P^{-\ome}|\bdel|, P^{-\tet}|\gam|\} = 1.
	\end{align}
	With this choice, the tuple $(\phi, \bpsi)$ lies on the boundary of $\frN(\eta)$, where we take $X=P$, and we discern from \eqref{4.2} and \eqref{4.4} with $a=0$, $\bb=\boldsymbol{0}$ and $q=1$ that
	\begin{align}\label{4.8}
		v_1(\gam, \bdel) \ll P^{-s} |T(\phi, \bpsi; P)| + P^{-1}(1 + |\gam| + |\bdel|) \ll P^{-s} |T(\phi, \bpsi; P)| + P^{-1+1/A}.
	\end{align}
	On the other hand, we find from the minor arcs bound \eqref{3.7} via \eqref{4.6} and \eqref{4.7} that
	\begin{align}\label{4.9}
		|T(\phi, \bpsi;P)| \ll P^{s-\kap \eta + \eps} \ll P^{s-{\textstyle{(\frac{\rho(d-1)}{\kap} + \frac{\sig}{t})}}^{-1}   \ome + \eps} \ll P^{s+\eps} |\bdel|^{-{\textstyle{(\frac{\rho(d-1)}{\kap} + \frac{\sig}{t})}}^{-1}}.
	\end{align}
	Alternatively, we may apply \eqref{3.9} and \eqref{4.7}, which leads us to the bound
	\begin{align}\label{4.10}
		|T(\phi, \bpsi;P)| \ll P^{s-\kap \eta + \eps} \ll P^{s-(t/\sig) \tet + \eps} \ll P^{s+\eps} |\gam|^{-t/\sig}.
	\end{align}	
	The statement of the lemma now follows upon combining the bounds \eqref{4.8}, \eqref{4.9} and \eqref{4.10}, and choosing $A$ sufficiently large.
\end{proof}

We may now complete the singular integral.

\begin{lemma}\label{L4.2}
	Suppose that \eqref{3.12} holds. Then the limit $\displaystyle{\lim_{T \to \infty}\frJ(T)}$ exists, and we have $$|\frJ(2T) - \frJ(T)| \ll T^{-\nu}$$
for some $\nu > 0$.
\end{lemma}
\begin{proof}
	Noting that $\vol\{\bdel \in \R^\rho: |\bdel| = \del_0\} \ll \del_0^{\rho-1}$, we see from Lemma~\ref{L4.1} that
	\begin{align*}
		\frJ(T) &\ll \int_{|\gam|, |\bdel| \le T} \min\left\{ 1, |\gam|^{-t/\sig+\eps}, |\bdel|^{-\big({\begin{textstyle}\frac{\rho(d-1)}{\kap} + \frac{\sig}{t}\end{textstyle}}\big)^{-1} + \eps}\right\} \d \gam \d \bdel\\
		&\ll 1 + \int_1^T \int_1^T \gam^{-\lam t/\sig + \eps} \del_0^{-(1-\lam)\big({\begin{textstyle}\frac{\rho(d-1)}{\kap} + \frac{\sig}{t}\end{textstyle}}\big)^{-1} + \rho - 1 + \eps} \d \gam \d \del_0
	\end{align*}
	for any $\lam \in [0,1]$. Here we note that the inequalities $\sig/t < 1$ and $\rho^2(d-1)/\kap + \rho \sig/t < 1$, both of which follow from \eqref{3.12}, suffice to bound the contribution from $[0,1]\times [1,T]$ and $[1,T] \times [0,1]$. The integrals over $[1,T]$ converge as $T \to \infty$, with the bound in the statement of the lemma, if and only if
	\begin{align*}
		\frac{\sig}{t} <\lam \qquad \text{ and } \qquad \frac{\rho^2(d-1)}{\kap} + \frac{\rho \sig}{t} < 1-\lam
	\end{align*}
	and $\eps$ is small enough. Clearly, such a $\lam$ exists whenever
	\begin{align*}
		\frac{\rho^2(d-1)}{\kap} + \frac{\sig(\rho+1)}{t} < 1,
	\end{align*}
	and this condition is strictly weaker than \eqref{3.12}.
\end{proof}

We next investigate the singular series.
\begin{lemma}\label{L4.3}
		For any $q \in \N$ and any $(a, \bb) \in (\Z/q\Z)^{\rho+1}$ with $(q, a, \bb)=1$ we have
		\begin{align*}
			q^{-s}|S(q; a, \bb)| \ll q^{\eps} \min \left\{ \left(\frac{q}{(q,a)}\right)^{\! -t/\sig} \! \! , \, q^{- \big({\textstyle{ \frac{\rho(d-1)}{\kap} + \frac{\sig}{t} }}  \big)^{-1}} \right\}.
		\end{align*}
\end{lemma}
\begin{proof}
	We may assume without loss of generality that $q>1$ and that $(a,q)<q$, since otherwise the bound is trivial. Fix a large number $A$ to be determined later and set $P=q^A$.
Fix $\tet$ such that $q/(q,a) = P^{\tet}$. Then we have $a/q \in \frM(\tet)$, with $X=P$, and we discern from \eqref{4.2} and \eqref{4.4} with $\gam=0$ and $\bdel = \boldsymbol 0$ that
	\begin{align}\label{4.11}
		q^{-s} |S(q; a, \bb)| \ll P^{-s} |T(a/q, \bb/q;P)| + qP^{-1} \ll P^{-s} |T(a/q, \bb/q;P)|+q^{1-A}.
	\end{align}
	On the other hand, $a/q$ lies just on the edge of the major arcs $\frM(\tet)$ in the sense that $a/q \not\in \frM(\tet-\eps)$ for any $\eps>0$. It follows that the minor arcs bound \eqref{H1} is applicable and yields
	\begin{align*}
		|T(a/q, \bb/q;P)| \ll P^{s- (t/\sig)\tet + \eps} \ll P^{s+\eps} \left(\frac{q}{(q,a)}\right)^{\! -t/\sig}.
	\end{align*}
	Upon inserting this bound into \eqref{4.11} and taking $A$ sufficiently large, we infer that
	\begin{align*}
		q^{-s}|S(q;a,\bb)| \ll P^\eps \left(\frac{q}{(q,a)}\right)^{\! -t/\sig} + q^{1-A} \ll P^\eps \left(\frac{q}{(q,a)}\right)^{\! -t/\sig}.
	\end{align*}
	
	Similarly, we can fix $\eta$ such that, upon recalling \eqref{4.6}, we have $q=P^{\ome}$. As before, this choice has the effect that $(a/q, \bb/q)$ marginally lies on the major arcs $\frN(\eta)$, again with $X=P$, and we are able to use the minor arcs bound \eqref{3.7}. Together with \eqref{4.6}, we thus arrive at the bound
	\begin{align*}
		|T(a/q, \bb/q;P)| \ll P^{s-\kap \eta + \eps} \ll P^{s+\eps} q^{- \big({\textstyle{ \frac{\rho(d-1)}{\kap} + \frac{\sig}{t} }}  \big)^{-1} }
	\end{align*}
	which, inserted into \eqref{4.11}, leads to the desired conclusion whenever $A$ has been taken sufficiently large.	
\end{proof}

With the help of Lemma~\ref{L4.3} we can now complete the singular series.
\begin{lemma}\label{L4.4}
	Suppose that \eqref{3.12} holds.
	Then the limit $\displaystyle{\lim_{T \to \infty} \frS(T)}$ exists, and we have $$|\frS(2T) - \frS(T)| \ll T^{-\nu}$$ for some $\nu > 0$.
\end{lemma}
\begin{proof}
	For any $\lam \in [0,1]$, Lemma~\ref{L4.3} shows that
	\begin{align*}
		\frS(T) &\ll \sum_{q \le T} q^{\eps} \sum_{\substack{a, \bb = 1 \\ (q, a, \bb)=1}}^q \left(\frac{q}{(q,a)}\right)^{- \lam t/\sig} q^{-(1-\lam)  {\begin{textstyle}( \frac{\rho(d-1)}{\kap} + \frac{\sig}{t} ) \end{textstyle}}^{-1}} \\
	& \ll \sum_{q \le T} q^{\rho+\eps -(1-\lam)  {\begin{textstyle}( \frac{\rho(d-1)}{\kap} + \frac{\sig}{t}) \end{textstyle}}^{-1} } \sum_{e | q}e^{1- \lam t/\sig}.
	\end{align*}
	These sums converge for sufficiently small $\eps$ if $\lambda$ is such that
	\begin{align*}
		\frac{\sig}{t} < \lam \qquad \text{and} \qquad (\rho+1) \left(\frac{\rho(d-1)}{\kap} + \frac{\sig}{t}\right) < 1-\lam.
	\end{align*}
	Such a $\lambda$ exists whenever \eqref{3.12} is satisfied. We conclude that under such conditions the singular series converges absolutely, with the bound in the statement of the lemma. 
\end{proof}

We now set
\begin{align*}
	\frJ = \lim_{T \to \infty} \frJ(T) \qquad \text{and} \qquad \frS = \lim_{T \to \infty} \frS(T).
\end{align*}
It then follows from \eqref{3.15} and \eqref{4.5} via Lemmata \ref{L4.2} and \ref{L4.4} that, under the hypotheses of Proposition \ref{P3.1}, we have an asymptotic formula of the shape
\begin{align*}
	N_{F, \bG}(X) = X^{s-k-\rho d} \frJ \frS + O(X^{s-k-\rho d - \nu})
\end{align*}
for some positive number $\nu$. It follows by standard arguments (see {\it e.g.}~ \cite{V:HL}, Section 2.6) that the singular series can be developed into an Euler product. For a natural number $q$ write $\Gam(q)$ for the number of solutions $\bx \in (\Z/q\Z)^s$ of the system of congruences
\begin{align*}
	F(\bx) \equiv 0 \pmod q  \qquad \text{and} \qquad G_i(\bx) \equiv 0 \pmod q \quad (1 \le i \le \rho),
\end{align*}
and set
\begin{align*}
	\chi_p = \lim_{h \to \infty} (p^h)^{\rho+1-s} \Gam(p^h).
\end{align*}
The arguments of Theorem~2.4 and Lemma~2.5 in \cite{V:HL} then show, \emph{mutatis mutandis}, that $\frS = \prod_p \chi_p$, and furthermore that there exists an integer $p_0$ having the property that
\begin{align*}
\frac12 < \prod_{p > p_0} \chi_p < \frac32.
\end{align*}
Hence we infer from an application of Hensel's lemma that $\frS>0$ if and only if the system \eqref{1.6} has a non-singular solution in all $p$-adic fields.

In a similar way, the arguments of Lemma~2 and Section 11 in \cite{Sch:82quad} show that $\frJ>0$ whenever the equations \eqref{1.6} have a non-singular solution in the real unit cube $[-1,1]^s$. Indeed, $\frJ$ and $\chi_p$ have an interpretation as the volume of the solution set of \eqref{1.6} in the real and $p$-adic unit cubes, respectively (see {\it e.g.} the discussion in \cite[Section 3]{Sch:85}). Upon setting $\chi_{\infty}=\frJ$, this completes the proof of Proposition \ref{P3.1}.

\section{The endgame}\label{S5}

For the proofs of Theorems \ref{T1.A} and \ref{T1.B} it remains to apply the work of Section \ref{S2} in order to find a suitable parameter tuple $(s, t, t_0, \sig,\Del)$ for which the Central Hypothesis as well as the hypotheses of Proposition \ref{P3.1} are satisfied. In particular, we will fix values for $t_0$, $\sig$, and $\Del$; the conditions will then be met whenever $s$ and $t$ are sufficiently large in terms of these parameters and $s$ is sufficiently large compared to $t$. We thus obtain a lower bound on $s$, which then yields the statements of the theorems.

We begin with the proof of Theorem~\ref{T1.B}. Suppose that $s=nu$, and for any vector $\bx \in \Z^s$ write $\bx = (\bx^{(1)}, \dots, \bx^{(u)})$ with $\bx^{(i)} \in \Z^n$ for $1 \le i \le u$. Let $\psi_1, \dots, \psi_u \in \Z[x_1, \dots, x_n]$ and set
\begin{align*}
	F(\bx)=\psi_1(\bx^{(1)}) + \ldots + \psi_u(\bx^{(u)}).
\end{align*}
For future reference we record the trivial inequality
\begin{align}\label{5.1}
	|a_1 \cdots a_w| \le |a_1|^w + \ldots + |a_w|^w,
\end{align}
which is valid for all $a_1, \dots, a_w \in \C$. When $\by \in \Z^l$ we write $\partial_{\by}$ for the forward difference operator, which acts on a polynomial $H \in \Z[x_1, \dots, x_l]$ via the relation
\begin{align}\label{5.2}
	\partial_{\by}H(\bx) = H(\bx+\by)-H(\bx).
\end{align}
For any $\underline{\bh} = (\bh_1, \dots, \bh_{d+1}) \in \Z^{n(d+1)}$, we further set $\vphi_i(\bx; \underline{\bh}) = \partial_{\bh_1} \cdots \partial_{\bh_{d+1}} \psi_i(\bx)$ for all $i$ with $ 1 \le i \le u$. Thus, by standard arguments (see {\it e.g.} Lemma~2.3 in \cite{V:HL}) we see via \eqref{5.1} that
\begin{align}\label{5.3}
	|T(\alp, \bbet)|^{2^{d+1}} \ll X^{(2^{d+1}-d-2)s} \sum_{1 \le i \le u} \Bigg| \sum_{|\underline{\bh}^{(i)}| \le 2X} \sum_{\bx^{(i)}} e(\alp \vphi_i(\bx^{(i)}; \underline{\bh}^{(i)}))\Bigg|^u,
\end{align}
where the inner sum runs over suitable subsets $I(\underline{\bh}^{(i)}) \subseteq \Z^n \cap [-X,X]^n$. Here, one should think of the $\bh_l^{(i)}$ as differencing variables of the shape $\bh_l^{(i)} = \by_l^{(i)}-\bx^{(i)}$, where all the $\by_l^{(i)}$ run over $\Z^n \cap [-X,X]^n$. Thus, the vector $(\bx^{(i)}, \underline{\bh}^{(i)})$ is the image of a vector $(\bx^{(i)}, \underline{\by}^{(i)}) \in [-X,X]^{(d+2)n}$ under a unimodular linear transformation.
Note also that $\partial_{\bh_1}\cdots \partial_{\bh_{d+1}}G_i(\bx)=0$ for $1 \le i \le \rho$, since each application of the differencing operator reduces the $\bx$-degree of the differenced polynomial by one.

Upon setting $u=2^{d+1} v$, $m=(d+2)n$ and
\begin{align*}
	f_{k,m}^{(i)}(\alp) = \sum_{|\underline{\bh}| \le 2X}\sum_{\bx \in I(\underline{\bh})} e(\alp \vphi_i(\bx; \underline{\bh})),
\end{align*}
we infer from \eqref{5.3} that
\begin{align}\label{5.4}
	|T(\alp, \bbet)| \ll X^{s-mv}  |f_{k,m}^{(i)}(\alp)|^v 
\end{align}
for some $i \in \{1, \dots, u\}$. The bound \eqref{H1} now follows easily from Theorem~\ref{T2.1}\eqref{T2.1b} with $t=v$ and $\sig = \sig_0(k)$. In a similar way, when $v = v_1 + 2 v_0$ with $v_0 \ge s_0(k)$, we obtain from \eqref{5.4} and Theorem~\ref{T2.1} that
\begin{align*}
	 \int_{\frm_k(Q)} |T(\alp, \bbet)| \d \alp 
	& \ll X^{s-mv} \sup_{\alp \in \frm_k(Q)} |f_{k,m}^{(i)}(\alp)|^{v_1}  \int_0^1 |f_{k,m}^{(i)}(\alp)|^{2v_0} \d \alp \\
	&\ll X^{s-mv} X^{m v_1+ \eps } Q^{-v_1/\sig_0(k)} X^{2mv_0-k+\eps}\\
	&\ll X^{s-k+\eps}Q^{-v_1/\sig_0(k)}
\end{align*}
for every $\bbet \in [0,1)^{\rho}$. This confirms the bound \eqref{H2} for the parameters $t_0=v_0 \ge s_0(k)$, $t=v$,  $s= 2^{d+1}nt$, $\sig = \sig_0(k)$ and $\Del=0$, upon taking $Q=X^{\tet}$. Finally, it follows from applying the trivial bound $f_{k,m}(\alp) \ll X^m$ that, should a tuple $(s,t,t_0, \sig, \Del)$ satisfy the Central Hypothesis, then the same is true when $s$ is replaced by any $s'$ with $s'>s$.
We summarise our results as follows.
\begin{lemma}\label{L5.1}
	The Central Hypothesis holds with the parameters $\Del=0$, $\sig=\sig_0(k)$, $t_0=s_0(k)$, $t > 2s_0(k)$ and $s = 2^{d+1}nt$.
\end{lemma}	
Theorem~\ref{T1.B} now follows upon combining Proposition \ref{P3.1} and Lemma~\ref{L5.1}.
In particular, the bounds \eqref{3.1}, \eqref{3.2} and \eqref{3.3} are certainly satisfied whenever we have $s-\dim \calV^*(\bG) > \max\{m_1, m_2, m_3\}$, where
\begin{align*}
	m_1 &= 2^{d+1}n( 2s_0(k) + \rho d \sig_0(k)), \\
	m_2 &= 2^{d-1}d \rho + 2^{d+1}n(\rho d + 2)\sig_0(k), \\
	m_3 &= 2^{d-1}\rho(\rho+1)(d-1) + 2^{d+1}n(\rho +2) \sig_0(k),
\end{align*}
where $\sig_0(k)$ and $s_0(k)$ are as in Theorem \ref{T2.1}.
We check easily that $s_0(k) \le  \sig_0(k)$ for all $k$ and so  $m_1 \le m_2$ for all values of $\rho$. In a similar manner, we see that $m_2 \ge m_3$ if and only if $\rho \le 4n\sig_0(k) + 1/(d-1)$. This proves Corollary~\ref{C1.10}.

For the proof of Theorem~\ref{T1.A} we follow a similar strategy, although in this case a more careful treatment allows us to consider the differencing variables that arise in the initial application of Weyl's inequality essentially as constants. 

Let
\begin{align*}
	F(\bx) = a_1 x_1^k + \ldots + a_s x_s^k,
\end{align*}
and for $\bh \in \Z^d$ set $\vphi_{\bh}(x) = \partial_{h_1} \cdots \partial_{h_d} x^k$. Thus, the polynomial $\vphi_{\bh}$ is of the shape
\begin{align}\label{5.5}
	\vphi_{\bh}(x) = h_1 \cdots h_d p_{\bh}(x),
\end{align}
where $p_{\bh}$ is a polynomial of degree $k-d$ whose leading coefficient is independent of $\bh$. Also, for $\bh_1, \dots, \bh_d \in \Z^s$ put $\bh^{(i)} = (h_{1,i}, \dots, h_{d,i})$ ($1 \le i \le s$). Just as before, Lemma~2.3 of \cite{V:HL} shows that for suitable sets $I(\bh^{(i)}) \subseteq [-X,X] \cap \Z$ one has
\begin{align}\label{5.6}
	|T(\alp, \bbet)|^{2^d} \ll X^{(2^d-d-1)s} \sum_{\substack{|\bh_j| \le X \\ (1 \le j \le d)}} \Bigg| \sum_{\substack{x_i \in I(\bh^{(i)})\\(1 \le i \le s)}} e\big(\alp(a_1 \vphi_{\bh^{(1)}}(x_1) + \ldots + a_s \vphi_{\bh^{(s)}} (x_s))\big)\Bigg|.
\end{align}
For $\bh \in \Z^d$ and $\calB \subseteq \Z^d$ we set
\begin{align*}
	f(\alp; \bh) = \sum_{x \in I(\bh)} e(\alp \vphi_{\bh}(x)) \qquad \text{ and } \qquad g(\alp; \calB) = \sum_{\bh \in \calB}|f(\alp; \bh)|,
\end{align*}
and we abbreviate $g(\alp ; [-X,X]^d\cap \Z^d)$ to $g(\alp)$. With this notation, we may rewrite the sum in \eqref{5.6}, and upon appealing to \eqref{5.1} we find that
\begin{align}\label{5.7}
	|T(\alp, \bbet)|^{2^d} &\ll X^{(2^d-d-1)s} \prod_{i=1}^s g(\alp a_i) \ll X^{(2^d-d-1)s} \sum_{i=1}^s g(\alp a_i)^s.
\end{align}
In view of \eqref{5.5}, the sum $f(\alp; \bh)$ is clearly trivial if $h_1 \cdots h_d=0$, so we need to remove all terms having $h_i=0$ for any $i$ from the exponential sums occurring within \eqref{5.7}. Set
\begin{align*}
	\calA_0 = \left\{ \bh \in [-X,X]^d \cap \Z^d : h_1 \cdots h_d=0\right\} \qquad \text{ and }\qquad \calA_1 = [-X,X]^d \cap \Z^d \setminus \calA_0.
\end{align*}
Then we have $|\calA_0| \ll X^{d-1}$ and thus
\begin{align*}
	g(\alp) \ll X^{d-1} \max_{\bh \in \calA_0} |f(\alp; \bh)| + g(\alp; \calA_1) \ll X^d + g(\alp; \calA_1).
\end{align*}
Hence, the bound in \eqref{5.7} becomes
\begin{align*}
	|T(\alp; \bbet)|^{2^d} &\ll X^{(2^d-d-1)s} \sum_{i=1}^s (X^d + g(\alp a_i; \calA_1))^s\\
	&\ll X^{(2^d-1)s} + X^{(2^d-d-1)s} \sum_{i=1}^s g(\alp a_i; \calA_1)^s.
\end{align*}

Upon setting $s=2^d t$, it follows that for some $i \in \{1, \ldots, s\}$ one has
\begin{align*}
	|T(\alp, \bbet)| \ll X^{s-t} + X^{s-(d+1)t} \left(\sum_{\bh \in \calA_1} |f(\alp a_i; \bh)|\right)^t.
\end{align*}
From Lemma~\ref{L2.3} with parameters $j=k-d$ and $r=d$, we infer that 
\begin{align*}
	\sup_{\alp \in \frm_k(Q)} \sum_{\bh \in \calA_1} |f(\alp a_i; \bh)| \ll X^{d+1+\eps}Q^{-1/\sig_0(k-d)} .
\end{align*}
This establishes \eqref{H1} with $\sig = \sig_0(k-d)$ as given in Theorem~\ref{T2.1}. In a similar manner, setting $t=v_1+2v_0$ with $v_0 \ge s_0(k-d)$ we find that
\begin{align*}
	& \int_{\frm_{k}(Q)} |T(\alp, \bbet)| \d \alp  \\
	&\ll X^{s-t} + X^{s-(d+1)t} \left( \sup_{\alp \in\frm_{k}(Q)} \sum_{\bh \in \calA_1} |f(\alp a_i; \bh)|\right)^{\! v_1} X^{2v_0 d} \max_{\bh \in \calA_1}\int_0^1 |f(\alp; \bh)|^{2v_0} \d \alp \\
	&\ll X^{s-t} + X^{s-(d+1)t} X^{v_1(d+1) + \eps}Q^{-{\textstyle\frac{v_1}{\sig_0(k-d)}}} X^{2 v_0 d} X^{2v_0-(k-d)+\eps}\\
	&\ll X^{s-k+d+\eps}Q^{-{\textstyle\frac{v_1}{\sig_0(k-d)}}}
\end{align*}
for all $\bbet \in [0,1)^{\rho}$. This establishes \eqref{H2} with parameters $\Del=d$, $\sig=\sig_0(k-d)$, $t_0=s_0(k-d)$, $t>2t_0$ and $s=2^dt$, upon taking $Q=X^{\tet}$. Again, we summarise our results.
\begin{lemma}\label{L5.2}
	Suppose that $k>d$. The Central Hypothesis holds with parameters $\Del=d$, $\sig = \sig_0(k-d)$, $t_0=s_0(k-d)$, $t > 2t_0$ and $s=2^dt$, where $\sig_0(k-d)$ and $s_0(k-d)$ are as in Theorem~$\ref{T2.1}$.
\end{lemma}
As before, by bounding any surplus exponential sums trivially, we see that the validity of the Central Hypothesis for a tuple $(s,t,t_0,\sig,\Del)$ implies its validity when $s$ is replaced by any $s'>s$.  Theorem~\ref{T1.A} follows upon combining Lemma~\ref{L5.2} with Proposition  \ref{P3.1}.

It remains to discuss under what conditions the local factors are positive in the case of Theorem \ref{T1.2}, where $d=2$ and $k=3$. The intersection of two smooth hypersurfaces has at most isolated singularities, so after possibly intersecting with a generic hyperplane we may assume the intersection variety to be smooth. This is the variety we work over in our ensuing deliberations. From Leep's work \cite[Corollary 2.4(ii)]{Leep84} it is clear that when $s \ge 23$, the quadratic form always vanishes on a $\Q_p$-linear space of dimension $10$. It then follows from Lewis' result \cite{L:52} on cubic forms in $10$ variables that the cubic, when restricted to this linear space, will have a non-trivial $p$-adic zero, which is a non-singular point of the intersection variety by construction. Thus, the singular series is positive. It thus suffices to show that the singular integral is also positive under the conditions given in Theorem \ref{T1.2}. Again, we need to show that the intersection variety contains a non-singular point over $\R$. When the matrix underlying the quadratic form has at least two positive and two negative eigenvalues, one easily shows that the quadratic form vanishes on a real line, and the cubic, restricted to that line, has at least one real zero, which again is non-singular by construction. This completes the proof of Theorem \ref{T1.2}.

\bibliographystyle{amsplain}
\bibliography{fullrefs}
\end{document}